\documentclass[12pt,a4paper]{amsart}
\usepackage{amssymb}
\usepackage{amsthm}
\usepackage{amsmath}
\usepackage{amsxtra}
\usepackage{latexsym}
\usepackage{mathrsfs}
\usepackage[all,cmtip]{xy}
\usepackage[all]{xy}
\usepackage{enumitem}
\usepackage{xcolor}
\usepackage{graphicx}
\usepackage{comment}

\usepackage[dvipdfmx]{hyperref}

\newtheorem{thm}{Theorem}[section]
\newtheorem{lem}[thm]{Lemma}
\newtheorem{conj}[thm]{Conjecture}
\newtheorem{claim}[thm]{Claim}
\newtheorem{prop}[thm]{Proposition}

\theoremstyle{definition}
\newtheorem{defn}[thm]{Definition}
\newtheorem{rmk}[thm]{Remark}
\newtheorem*{ack}{Acknowledgement}
\numberwithin{equation}{section}

\def\C{{\mathbb C}}
\def\Q{{\mathbb Q}}
\def\R{{\mathbb R}}
\def\Z{{\mathbb Z}}
\def\P{{\mathbb P}}

\def\Hom{\mathop{\mathrm{Hom}}\nolimits}
\def\QQ{\overline{\mathbb Q}}

\DeclareMathOperator{\pr}{pr}

\DeclareMathOperator{\id}{id}

\DeclareMathOperator{\End}{End}

\DeclareMathOperator{\Exc}{Exc}
\DeclareMathOperator{\Supp}{Supp}

\DeclareMathOperator{\CH}{CH}

\newenvironment{claimproof}[0]
  {%
   \paragraph{\it Proof.}%
  }
  {%
    \hfill$\blacksquare$%
  }

\newenvironment{notation}[0]{%
  \begin{list}%
    {}%
    {\setlength{\itemindent}{0pt}
     \setlength{\labelwidth}{4\parindent}
     \setlength{\labelsep}{\parindent}
     \setlength{\leftmargin}{5\parindent}
     \setlength{\itemsep}{0pt}
     }%
   }%
  {\end{list}}

\title[]
{Arithmetic and dynamical degrees of self-morphisms of semi-abelian varieties}
\author{Yohsuke Matsuzawa}
\author{Kaoru Sano}
\address{Graduate school of Mathematical Sciences, the University of Tokyo, Komaba, Tokyo,
153-8914, Japan}
\address{Department of Mathematics, Faculty of Science, Kyoto University, Kyoto 606-8502, Japan}
\email{myohsuke@ms.u-tokyo.ac.jp}
\email{ksano@math.kyoto-u.ac.jp}
\begin{document}

\begin{abstract}
We prove a conjecture by Kawaguchi-Silverman on arithmetic  and dynamical degrees,
for self-morphisms of semi-abelian varieties.
Moreover, we determine the set of the arithmetic degrees of orbits and
the (first) dynamical degrees of self-morphisms of semi-abelian varieties.
\end{abstract}

\maketitle

\tableofcontents 

\section{Introduction}

Let $X$ be a smooth projective variety and $f \colon X \dashrightarrow X$ a rational self-map,
both defined over $ \overline{\mathbb Q}$.
Silverman introduced the notion of arithmetic degree in \cite{sil},
which measures the arithmetic complexity of $f$-orbits.
Fix a Weil height function $h_{X}$ associated with an ample divisor on $X$.
Consider a point $x\in X( \overline{\mathbb Q})$ such that for all $n\geq0$, $f^{n}(x)$ is not contained in the indeterminacy locus of $f$.
The {\it arithmetic degree} of $f$ at $x$ is
\[
\alpha_{f}(x)=\lim_{n \to \infty}\max\{h_{X}(f^{n}(x)),1\}^{1/n}
\]
provided that the limit exists.

In \cite{ks1}, Kawaguchi and Silverman proved that when $f$ is a morphism,
the arithmetic degree $ \alpha_{f}(x)$ always exists and is equal to either $1$ or the absolute value of 
one of the eigenvalues of $f^{*} \colon N^{1}(X) {\otimes}_{\Z}\R \longrightarrow N^{1}(X) {\otimes}_{\Z}\R$,
where $N^{1}(X)$ is the group of divisors modulo numerical equivalence.

When $f$ is dominant, it is conjectured in  \cite{sil}, \cite[Conjecture 6]{ks3} that the arithmetic degree of any Zariski dense orbits
are equal to the first dynamical degree $\delta_{f}$ of $f$ (cf.  Conjecture \ref{ksc}). 
This is the Kawaguchi-Silverman conjecture, and we abbreviate it as KSC.
Here, the first dynamical degree is a birational invariant of $f$ which measures the geometric complexity of
the dynamical system.
For a surjective morphism $f$, $\delta_{f}$ is equal to the spectral radius of the linear map
$f^{*} \colon N^{1}(X) {\otimes}_{\Z}\R \longrightarrow N^{1}(X) {\otimes}_{\Z}\R$.

Let $A(f)$ be the set of the arithmetic degrees of $f$, i.e. 
\[A(f)=\{ \alpha_{f}(x) \mid x\in X \text{ with $f^{n}(x)\notin I_{f}$ for every $n\geq0$}\}\]
where $I_{f}$ is the indeterminacy locus of $f$.
Determining the set $A(f)$ for a given $f$ is an interesting problem.
In \cite[Theorem 1.6]{mss}, Shibata and we proved that for any surjective morphism $f$, there exists a point $x \in X$ such that
$ \alpha_{f}(x)=\delta_{f}$.
When $X$ is a toric variety and $f$ is a self-rational map on $X$ that is induced by
a group homomorphism of the algebraic torus, the set $A(f)$ is completely determined \cite{sil,lin}.

When $X$ is quasi-projective, the arithmetic degrees and dynamical degrees can be defined by taking a smooth compactification of $X$.
In this paper, we prove KSC for self-morphisms of semi-abelian varieties and determine the set $A(f)$.

\begin{thm}\label{main semiab}
Let $X$ be a semi-abelian variety and $f \colon X \longrightarrow X$ a self-morphism (not necessarily surjective), both defined over $ \overline{\mathbb Q}$.
\begin{enumerate}
\item[\rm (1)] 
For every $x\in X(\QQ)$, the arithmetic degree $ \alpha_{f}(x)$ exists.
Moreover, write $f=T_{a}\circ g$ where $T_{a}$ is the translation by a point $a\in X(\QQ)$ and $g$ is a group homomorphism.
Then $A(f)=A(g)$.
\item[\rm (2)] Suppose $f$ is surjective. Then for any point $x \in X(\QQ)$ with Zariski dense $f$-orbit, we have $ \alpha_{f}(x)=\delta_{f}$. 
\item[\rm (3)] Suppose $f$ is a group homomorphism. Let $F(t)$ be the monic minimal polynomial of $f$ as an element of $\End(X) {\otimes}_{\Z}\Q$ and
\[
F(t)=t^{e_{0}}F_{1}(t)^{e_{1}}\cdots F_{r}(t)^{e_{r}}
\]
the irreducible decomposition in $\Q[t]$ where $e_{0}\geq0$ and  $e_{i}>0$ for $i=1,\dots, r$. 
Let $\rho(F_{i})$ be the maximum among the absolute values of the roots of $F_{i}$.
Then we have
\[
A(f)\subset\{1, \rho(F_{1}), \rho(F_{1})^{2},\dots, \rho(F_{r}), \rho(F_{r})^{2} \}.
\]
More precisely, set
\[
X_{i}=f^{e_{0}}F_{1}(f)^{e_{1}}\cdots F_{i-1}(f)^{e_{i-1}} F_{i+1}(f)^{e_{i+1}}\cdots F_{r}(f)^{e_{r}}(X).
\]
Define
\[
A_{i}= \begin{cases}
	\{\rho(F_{i})\} \qquad \text{if $X_{i}$ is an algebraic torus,} \\
	\{\rho(F_{i})^{2}\} \qquad \text{if $X_{i}$ is an abelian variety,}\\
	\{\rho(F_{i}), \rho(F_{i})^{2}\} \qquad \text{otherwise.}
\end{cases}
 \]
Then we have
\[
A(f)=\{1\}\cup A_{1}\cup \cdots \cup A_{r}.
\]

\end{enumerate}
\end{thm}

 \begin{rmk}\label{rmk:split}
Actually,  in the situation of Theorem \ref{main semiab} (3), $f$ is conjugate by an isogeny to a group homomorphism of the form
\[
f_{0}\times \cdots \times f_{r} \colon X_{0}\times \cdots \times X_{r} \longrightarrow X_{0}\times \cdots \times X_{r}
\]
where $A(f_{0})=\{1\}$ and $A(f_{i})=\{1\}\cup A_{i}$ for $i=1,\dots,r$. 
 \end{rmk}

We can characterize the set of points whose arithmetic degrees are equal to $1$ as follows
(cf. \cite{sano} for related results).

\begin{thm}\label{main2 semiab}
Let $X$ be a semi-abelian variety and $f \colon X \longrightarrow X$ a surjective morphism
both defined over $ \overline{\mathbb Q}$.
Write $f=T_{a}\circ g$  where $T_{a}$ is the translation by $a\in X(\QQ)$ and $g$ is a surjective group endomorphism of $X$.
Suppose that the minimal polynomial of $g$ has no irreducible factor that is
a cyclotomic polynomial.
Then there exists a point $b \in X(\QQ)$ such that, for any $x \in X(\QQ)$, the following are equivalent:
\begin{enumerate}
\item[\rm (1)] $\alpha_{f}(x)=1$;
\item[\rm (2)] $\#\ O_{f}(x)<\infty$;
\item[\rm (3)] $x \in b+X( \overline{\mathbb Q})_{\rm tors}$.
\end{enumerate}
Here $X( \overline{\mathbb Q})_{\rm tors}$ is the set of torsion points.
\end{thm}

\begin{rmk}
It is easy to see that when $f$ is a surjective group homomorphism, we can take $b=0$.
\end{rmk}

\begin{rmk}
If the minimal polynomial of $g$ has irreducible factor that is a cyclotomic polynomial, then
one of $f_{i}$ in Remark \ref{rmk:split} (applied to $f=g$) has dynamical degree $1$.
\end{rmk}

To prove the above theorems, we calculate the dynamical degrees
of self-morphisms of semi-abelian varieties.

\begin{thm}\label{main dyn deg}
Let $X$ be a semi-abelian variety over an algebraically closed field of characteristic zero.
\begin{enumerate}
\item[\rm (1)]
Let $f \colon X \longrightarrow X$ be a surjective group homomorphism.
Let 
\[
\xymatrix{
0 \ar[r] & T \ar[r] & X \ar[r]^{\pi} & A \ar[r] &0
}
\]
 be an exact sequence where $T$ is a torus and $A$ is an abelian variety.
 Then $f$ induces surjective group homomorphisms
 \begin{align*}
 f_{T}:=f|_{T}  & \colon T \longrightarrow T\\
 g & \colon  A \longrightarrow A
 \end{align*}
 with $g\circ \pi=\pi \circ f$.
 Then we have
 \[
 \delta_{f}=\max\{\delta_{g}, \delta_{f_{T}}\}.
 \]
Moreover, let $P_{T}$ and $P_{A}$ be the monic minimal polynomials of 
$f_{T}$ and $g$ as elements of $\End(T)_{\Q}$ and $\End(A)_{\Q}$ respectively.
Then, $\delta_{f_{T}}=\rho(P_{T})$ and $\delta_{g}=\rho(P_{A})^{2}$.
\item[\rm (2)]
Let $f \colon X \longrightarrow X$ be a surjective group homomorphism and
$a \in X$ a closed point. Then $\delta_{T_{a}\circ f}=\delta_{f}$.
\end{enumerate}
\end{thm}

\begin{rmk}\label{dyn deg of monomial map}
The description of $\delta_{f_{T}}$ in Theorem \ref{main dyn deg}(1) is well-known
(see for example \cite{sil}).
\end{rmk}

The outline of  this paper is as follows.
In \S \ref{sec:notation and preliminaries}, we fix our notation, give the definitions of the dynamical and arithmetic degrees, and
summarize their basic properties for later use.
In \S \ref{splitting}, we prove a lemma that says every group homomorphism of a semi-abelian variety
``splits into rather simple ones''.
In \S \ref{sec:abelian varieties}, we prove our main theorems for isogenies of abelian varieties.
We use these to prove the main theorems.
In \S \ref{subsec:semi-abelian dyn deg}, we calculate the first dynamical degrees of self-morphisms of semi-abelian varieties
and prove Theorem \ref{main dyn deg}.
In \S \ref{subsec:semi-abelian arith deg}, we prove Theorem \ref{main semiab} and \ref{main2 semiab}.

\section{Notation and Preliminaries}\label{sec:notation and preliminaries}

\subsection{Notation}

In this paper, the ground field is either $ \overline{\mathbb Q}$ or an arbitrary algebraically closed field of
characteristic zero.
A variety is a separated irreducible reduced scheme of finite type over an algebraically closed field $k$.
Let $X$ be a variety over $k$ and $f \colon X \dashrightarrow X$ a rational map.
We use the following notation:
\begin{notation}
\item[$\CH^{1}(X)$] The group of codimension one cycles on $X$ modulo rational equivalence is denoted by $\CH^{1}(X)$.
\item[$N^{1}(X)$] For a complete variety $X$, the group of divisors modulo numerical equivalence is denoted by $N^{1}(X)$.
\item[$I_{f}$] The indeterminacy locus of $f$ is denoted by $I_{f}$.
\item[$X_{f}(k)$ ] We write $X_{f}(k)=\{x\in X(k) \mid \text{$f^{n}(x)\notin I_{f}$ for every $n\geq0$}\}$.
\item[$O_{f}(x)$] The $f$-orbit of $x\in X_{f}$ is denoted by $O_{f}(x)$ i.e. $O_{f}(x)=\{f^{n}(x)\mid n\geq 0\}$.
\item[$h_{D}$ ] When $k= \overline{\mathbb Q}$, a Weil height function associated with a divisor $D$ is denoted by $h_{D}$. 
We refer to \cite{bg, hs} for the definitions and basic properties of Weil height functions.
\item[$\rho(T)$] For an endomorphism $T \colon V \longrightarrow V$ of a finite dimensional real vector space $V$,
the maximum among the absolute values of the eigenvalues of $T$ is called the spectral radius of $T$ and denoted by $\rho(T)$.
\item[$\rho(F)$] For a polynomial $F\in \C[t]$, the maximum among the absolute values of the roots of $F$ is denoted by
$\rho(F)$ and called the spectral radius of $F$.
\item[$M_{K}$] For a $\Z$-module $M$ and a field $K$, we write $M_{K}=M {\otimes}_{\Z}K$.
\item[$T_{a}$] Let $X$ be a commutative algebraic group and $a\in X(k)$ a point. The translation by $a$ is denoted by $T_{a}$.
\end{notation}

\subsection{Dynamical degrees}

Let $X, Y$ be smooth projective varieties defined over an algebraically closed field of characteristic zero
and $f \colon X \dashrightarrow Y$ a (not necessarily dominant) rational map.
We define pull-back $f^{*} \colon \CH^{1}(Y) \longrightarrow \CH^{1}(X)$ as follows.
Take a resolution of indeterminacy $\pi \colon X' \longrightarrow X$ of $f$ with $X'$ smooth projective.
Write $f'=f \circ \pi $.
Then we define $f^{*}=\pi_{*}\circ {f'}^{*}$.
This is independent of the choice of resolution.

\begin{defn}
Let $f \colon X \dashrightarrow X$ be a dominant rational map.
Fix an ample divisor $H$ on $X$.
Then the (first) dynamical degree of $f$ is
\[
\delta_{f}=\lim_{n \to \infty} ((f^{n})^{*}H\cdot H^{\dim X-1})^{1/n}.
\]
This is independent of the choice of $H$.
We refer to \cite{dang, df, tru0}, \cite[\S 4]{ds} for basic properties of dynamical degrees.
\end{defn}

\begin{rmk}\label{rmk on dyn deg}\ 

\begin{enumerate}
\item There are other definitions of the dynamical degree.
Fix a norm $\left\| \cdot \right\|$ on $\Hom(N^{1}(X)_{\R}, N^{1}(X)_{\R})$.
Then $ \delta_{f}=\lim_{n \to \infty}\left\|(f^{n})^{*}\right\|^{1/n}$.
When $f$ is a morphism, $\delta_{f}$ is the spectral radius of $f^{*} \colon N^{1}(X)_{\R} \longrightarrow N^{1}(X)_{\R}$.
If the ground field is $\C$, this is equal to the spectral radius of $f^{*} \colon H^{1,1}(X) \longrightarrow H^{1,1}(X)$
(cf. \cite[\S 4]{ds}).
\item Dynamical degree is a birational invariant. That is, if $\pi \colon X \dashrightarrow X'$ is a birational map
and $f \colon X \dashrightarrow X$ and $f' \colon X' \dashrightarrow X'$ are conjugate by $\pi$, then $\delta_{f}=\delta_{f'}$.
\item Let $X, Y$ be  smooth projective varieties and $f \colon X \dashrightarrow X$, $g \colon Y \dashrightarrow Y$ dominant rational maps.
Then, by definition, it is easy to see that $\delta_{f\times g}=\max\{ \delta_{f}, \delta_{g}\}$.
\end{enumerate}
\end{rmk}

\subsection{Arithmetic degrees}

In this subsection, the ground field is $ \overline{\mathbb Q}$.

\begin{defn}
Let $f \colon X \dashrightarrow X$ be a rational self-map of a smooth quasi-projective variety. 
\begin{enumerate}
\item[\rm (1)] A point $x\in X_{f}(\QQ)$ is called $f$-preperiodic if the orbit $O_{f}(x)=\{ f^{n}(x)\mid n\geq0\}$ is a finite set.
\item[\rm (2)] 
Fix a smooth projective variety $ \overline{X}$ and an open embedding $X \subset \overline{X}$.
Let $H$ be an ample divisor on $ \overline{X}$ and take a Weil height function $h_{H}$ associated with $H$.
The arithmetic degree $ \alpha_{f}(x)$ of $f$ at $x \in X_{f}(\QQ)$ is defined by
\[
\alpha_{f}(x)=\lim_{n \to \infty}\max\{1, h_{H}(f^{n}(x))\}^{1/n}
\]
if the limit exists.
Since the convergence of this limit is not proved in general, we introduce the following:
\begin{align*}
&\overline{\alpha}_{f}(x)=\limsup_{n \to \infty}\max\{1, h_{H}(f^{n}(x))\}^{1/n},\\
& \underline{\alpha}_{f}(x)=\liminf_{n \to \infty}\max\{1, h_{H}(f^{n}(x))\}^{1/n}.
\end{align*}
We call $ \overline{\alpha}_{f}(x)$ the upper arithmetic degree and $ \underline{\alpha}_{f}(x)$ the lower arithmetic degree.
 The definitions of the (upper, lower) arithmetic degrees do not depend on the choice of $ \overline{X}, H$ and $h_{H}$ (\cite[Proposition 12]{ks3} \cite[Theorem 3.4]{mss}). 
\item[\rm (3)] Suppose that $ \alpha_{f}(x)$ exists for every $x \in X_{f}(\QQ)$.
Then we write $A(f)=\{ \alpha_{f}(x) \mid x \in X_{f}(\QQ)\}$.
\end{enumerate}
\end{defn}
 
 \begin{rmk}
 By definition, $1 \leq \underline{\alpha}_{f}(x)\leq \overline{\alpha}_{f}(x)$. When $x$ is $f$-preperiodic, $ \alpha_{f}(x)=1$.
 \end{rmk}
 
 \begin{rmk}
 When $X$ is projective and $f$ is a morphism, $ \alpha_{f}(x)$ exists for every $x \in X$ \cite[Theorem 3]{ks1}. 
 \end{rmk}

\begin{lem}\label{product case}
Let $X, Y$ be smooth quasi-projective varieties and $f \colon X \dashrightarrow X$, $g \colon Y \dashrightarrow Y$
rational maps.
Let $x \in X_{f}(\QQ)$ and $y \in Y_{g}(\QQ)$.
If $ \alpha_{f}(x)$ and $ \alpha_{g}(y)$ exist, then $ \alpha_{f\times g}(x,y)$ also exists and
\[
\alpha_{f\times g}(x,y)=\max\{ \alpha_{f}(x), \alpha_{g}(y) \}.
\]
\end{lem}
\begin{proof}
It is enough to prove when $X, Y$ are projective.
Take ample divisors $H_{X}, H_{Y}$ on $X, Y$ respectively.
Fix associated height functions $h_{H_{X}}, h_{H_{Y}}$ so that $h_{H_{X}}\geq1$ and $h_{H_{Y}}\geq 1$.
Then $h:=h_{H_{X}}\circ \pr_{1}+h_{H_{Y}}\circ \pr_{2}$ is an ample height function on $X\times Y$.
Then 
\begin{align*}
&\lim_{n \to \infty}h((f\times g)^{n}(x,y))^{1/n}=\lim_{n \to \infty}(h_{H_{X}}(f^{n}(x))+h_{H_{Y}}(g^{n}(y)))^{1/n}\\
&=\max\{\lim_{n\to\infty}h_{H_{X}}(f^{n}(x))^{1/n},\lim_{n\to\infty}h_{H_{Y}}(g^{n}(y))^{1/n} \}
=\max\{ \alpha_{f}(x), \alpha_{g}(y)\}.
\end{align*}
\end{proof}

\begin{lem}\label{ascent}
Consider the following commutative diagram
\[
\xymatrix{
Y \ar@{-->}[r]^{g} \ar[d]_{\pi} & Y \ar[d]^{\pi}\\
X \ar@{-->}[r]_{f} & X
}
\]
where $X, Y$ are smooth quasi-projective varieties, $f,g$ rational maps
and $\pi$ a surjective morphism.
Let $y\in Y_{g}(\QQ)$ be a point such that $\pi(y)\in X_{f}(\QQ)$.
Then
\begin{align*}
\overline{\alpha}_{g}(y)\geq \overline{\alpha}_{f}(x)\\
\underline{\alpha}_{g}(y) \geq \underline{\alpha}_{f}(x).
\end{align*}
\end{lem}
\begin{proof}
We may assume $X, Y$ are projective.
Take an ample divisor $H_{X}$ on $X$ and fix an associated height function $h_{H_{X}}$ with $h_{H_{X}}\geq 1$.
Take an ample divisor $H_{Y}$ on $Y$ so that $H_{Y}-\pi^{*}H_{X}$ is ample.
Then we can take a height function $h_{H_{Y}}$ associated with $H_{Y}$ so that $h_{H_{Y}}\geq h_{H_{X}}\circ \pi$.
Then
\begin{align*}
\overline{\alpha}_{f}(x)=&\limsup_{n\to \infty}h_{H_{X}}(f^{n}(x))^{1/n}=\limsup_{n\to \infty}h_{H_{X}}(\pi(g^{n}(y)))^{1/n}\\
&\leq \limsup_{n\to \infty}h_{H_{Y}}(g^{n}(y))^{1/n}= \overline{\alpha}_{g}(y).
\end{align*}
The second inequality can be proved similarly. 
\end{proof}

\begin{lem}\label{comparing arithmetic deg}
Consider the following commutative diagram
\[
\xymatrix{
Y \ar@{-->}[r]^{g} \ar[d]_{\pi} & Y \ar[d]^{\pi}\\
X \ar@{-->}[r]_{f} & X
}
\]
where $X, Y$ are smooth projective varieties and $f,g$ rational maps.
Suppose there exists a non-empty open subset $U \subset X$ such that $\pi \colon V:=\pi^{-1}(U) \longrightarrow U$ is finite.

Let $y\in Y(\QQ)$ such that
$y \in Y_{g}(\QQ)$, 
$x:=\pi(y)\in X_{f}(\QQ)$, $O_{g}(y)\subset V$ and $O_{f}(x) \subset U$.
If $ \alpha_{g}(y)$ exists, then $ \alpha_{f}(x)$ also exists and $ \alpha_{g}(y)= \alpha_{f}(x)$.
\end{lem}
\begin{proof}
Take an ample divisor $H$ on $X$ and let $h_{H}$ be a height function associated with $H$.
We choose $h_{H}$ so that $h_{H}\geq 1$.
Then we have
\begin{align}\label{eq 0}
h_{H}(f^{n}(x))=h_{H}(\pi(g^{n}(y)))=h_{\pi^{*}H}(g^{n}(y)).
\end{align}
Here we choose $h_{\pi^{*}H}$ so that $h_{\pi^{*}H}=h_{H}\circ \pi$.

Let $\xymatrix{Y \ar[r]^{\beta} & Z \ar[r]^{\alpha} & X}$ be the Stein factorization of $\pi$.
Then $\beta \colon V \longrightarrow \alpha^{-1}(U)$ is an isomorphism.
Let $\nu \colon \widetilde{Z} \longrightarrow Z$ be a resolution of singularities that is a composite of
blow-ups with center in the singular locus of $Z$.
In particular, the blow-up centers do not intersect with $\alpha^{-1}(U)$.
Let $\mu \colon \widetilde{Y} \longrightarrow Y$ be a resolution of indeterminacy of $Y \dashrightarrow \widetilde{Z}$ that
is a composite of blow-ups along smooth centers outside $V$.
The situation is summarized in the following diagram:
\[
\xymatrix{
	& \widetilde{Y} \ar[dl]_{\mu} \ar[dr]^{ \widetilde{\beta}} &	\\
Y \ar[dd]_{\pi} \ar[dr]^{\beta} &	 & \widetilde{Z}  \ar[dl]^{\nu}\\
	&Z \ar[dl]^{\alpha} &	\\
X &	&	.
}
\]
Then, since $\nu$ is a sequence of blow-ups and $\alpha^{*}H$ is ample (because $\alpha$ is finite),
there exists an effective $\nu$-exceptional $\Q$-divisor $E_{\nu}$ on $ \widetilde{Z}$ such that 
\begin{align*}
\nu^{*}\alpha^{*}H-E_{\nu}
\end{align*}
is ample (cf.  \cite[II Proposition 7.10 (b)]{har}).  
Also, since $ \widetilde{Z}$ is smooth, there exists an effective $\Q$-divisor $E_{ \widetilde{\beta}}$ on $ \widetilde{Y}$ that is $ \widetilde{\beta}$-exceptional
such that
\begin{align*}
\widetilde{\beta}^{*}(\nu^{*}\alpha^{*}H-E_{\nu})-E_{ \widetilde{\beta}}
\end{align*}
is ample. 
(To see this, use \cite[Lemma 2.62]{kolmor} or
write $ \widetilde{\beta}$ as a blow up along an ideal and apply \cite[II Proposition 7.10 (b)]{har}.)
Set $A=\widetilde{\beta}^{*}(\nu^{*}\alpha^{*}H-E_{\nu})-E_{ \widetilde{\beta}}$ and
$E= \widetilde{\beta}^{*}E_{\nu}+E_{ \widetilde{\beta}}$.
Then
$A$ is ample, $E$ is effective, $\Supp E\cap \mu^{-1}(V)=\emptyset$ and
$\mu^{*}\pi^{*}H=A+E$.

Let $ \widetilde{y}=\mu^{-1}(y)$. Let $ \widetilde{g}=\mu^{-1}\circ g \circ \mu \colon \widetilde{Y} \dashrightarrow \widetilde{Y}$
be the rational map induced by $g$.
(Since $g(y)\in V$ and $\mu$ is isomorphic over $V$, $ \mu^{-1}\circ g\circ \mu$ is well-defined.)
Then by \cite[Theorem 3.4(i)]{mss}, $ \alpha_{ \widetilde{g}}( \widetilde{y})$ exists and equals $ \alpha_{g}(y)$ since $ \alpha_{g}(y)$ exists.
(Note that \cite[Theorem 3.4(i)]{mss} works for possibly non-dominant rational maps $f, g$.)
Choose height function $h_{A+E}$ so that $h_{A+E}=h_{\pi^{*}H}\circ \mu$.
Then
\begin{align}\label{eq 1}
h_{\pi^{*}H}(g^{n}(y))=h_{A+E}( \widetilde{g}^{n}( \widetilde{y})).
\end{align}
Since the $ \widetilde{g}$-orbit of $ \widetilde{y}$ does not intersect with $\Supp E$,
we have
\begin{align}\label{ineq 1}
h_{A+E}( \widetilde{g}^{n}( \widetilde{y})) \geq h_{A}( \widetilde{g}^{n}( \widetilde{y}))+O(1)
\end{align}
where $h_{A}$ is a height associated with $A$.
Here, we use the fact that any height function associated with an effective divisor 
is bounded below outside the support of the divisor (cf.  \cite[Theorem B.3.2.(e)]{hs}).
If $C>0$ is a positive number with $CA-(A+E)$ is ample, we have
\begin{align}\label{ineq 2}
h_{A+E}( \widetilde{g}^{n}( \widetilde{y})) \leq Ch_{A}( \widetilde{g}^{n}( \widetilde{y}))+O(1).
\end{align}

By (\ref{eq 0}), (\ref{eq 1}),  (\ref{ineq 1}) and  (\ref{ineq 2}), we have
\begin{align*}
h_{A}( \widetilde{g}^{n}( \widetilde{y}))+O(1) \leq h_{H}(f^{n}(x)) \leq Ch_{A}( \widetilde{g}^{n}( \widetilde{y}))+O(1).
\end{align*}
Since $ \alpha_{ \widetilde{g}}( \widetilde{y})=\lim_{n \to \infty}\max\{1, h_{A}( \widetilde{g}^{n}( \widetilde{y}))\}^{1/n}$
exists and is equal to $ \alpha_{g}(y)$,
we get $\lim_{n \to \infty} h_{H}(f^{n}(x))^{1/n}= \alpha_{g}(y)$.
\end{proof}

\subsection{Kawaguchi-Silverman conjecture}

In \cite{ks3}, Kawaguchi and Silverman formulated the following conjecture.

\begin{conj}\label{ksc}
Let $X$ be a smooth quasi-projective variety and $f \colon X \dashrightarrow X$ a dominant rational map,
both defined over $ \overline{\mathbb Q}$.
Let $x \in X_{f}( \overline{\mathbb Q})$.
\begin{enumerate}
\item[\rm (1)] The limit defining $\alpha_{f}(x)$ exists.
\item[\rm (2)] The arithmetic degree $\alpha_{f}(x)$ is an algebraic integer.
\item[\rm (3)] The collections of the arithmetic degrees $\{ \alpha_{f}(x) \mid x \in X_{f}( \overline{\mathbb Q}) \}$
is a finite set.
\item[\rm (4)] If the orbit $ O_{f}(x)=\{ f^{n}(x) \mid n=0,1,2, \dots \}$ is Zariski dense in $X$,
then $\alpha_{f}(x)=\delta_{f}$.
\end{enumerate}
\end{conj}

\begin{rmk}
In \cite{ks3}, the conjecture is formulated for smooth projective varieties.
Of course, quasi--projective version of the conjecture is equivalent to the projective version.
\end{rmk}

\begin{rmk}
For any dominant rational map $f$, the arithmetic degrees are bounded by the dynamical degree:
\[
\overline{\alpha}_{f}(x)\leq \delta_{f}
\]
for every $x\in X_{f}(\QQ)$ \cite{ks3}, \cite[Theorem 1.4]{ma}.

\end{rmk}

\begin{rmk}
If $X$ is projective and $f$ is a morphism, then Conjecture \ref{ksc} (1)(2)(3) are true \cite[Theorem 3]{ks1}.
\end{rmk}

\begin{rmk}
Conjecture \ref{ksc} is verified in several cases.
We refer to \cite[Remark 1.8, Theorem 1.3]{mss} for details.
\end{rmk}

The following results are used later.

\begin{thm}\cite[Theorem 4]{ks1},\cite[Theorem 4, Corollary 32]{sil},\cite[Theorem 2]{sil2}\label{key known thm}
\ 
\begin{enumerate}
\item[\rm (1)] For any self-morphisms of abelian varieties,  Conjecture \ref{ksc} is true.
\item[\rm (2)] Let $X$ be an algebraic torus and  $f \colon X \longrightarrow X$ be a group homomorphism. 
Then Conjecture \ref{ksc} is true for $f$.
Moreover, let $F(t)$ be the minimal monic polynomial of $f$ as an element of $\End(X)_{\Q}$ and 
$F(t)=t^{e_{0}}F_{1}(t)^{e_{1}}\cdots F_{r}(t)^{e_{r}}$ the irreducible decomposition.
Then $A(f)=\{1, \rho(F_{1}),\dots, \rho(F_{r})\}$.
\end{enumerate}
\end{thm}

\section{Splitting lemma}\label{splitting}

In this section, the ground field is an algebraically closed field of characteristic zero.
Let $X$ be a a semi-abelian variety,
i.e. a commutative algebraic group that is an extension of an abelian variety by an algebraic torus.
Note that $X$ is divisible i.e. the morphism $X \longrightarrow X; x \mapsto nx$ is surjective for every $n>0$.

\begin{lem}\label{split lem 0}
Let $f \colon X \longrightarrow X$ be a group homomorphism.
Let $F(t) \in \Z[t]$ be a polynomial such that $F(f)=0$ in $\End(X)$.
Suppose $F(t)=F_{1}(t)F_{2}(t)$ where $F_{1}, F_{1} \in \Z[t]$ are coprime in $\Q[t]$.
Set $X_{1}=F_{2}(f)(X)$ and $X_{2}=F_{1}(f)(X)$.
Then $X=X_{1}+X_{2}$ and $X_{1}\cap X_{2}$ is finite.
In other words, the morphism $X_{1}\times X_{2} \longrightarrow X; (x_{1},x_{2}) \mapsto x_{1}+x_{2}$
is an isogeny. 
\end{lem}
\begin{proof}
The proof of \cite[Lemma 3.1]{sil2} works for semi-abelian varieties.
\end{proof}
In the situation of Lemma \ref{split lem 0}, write $f_{i}=f|_{X_{i}}$.
Then $F_{i}(f_{i})=0$ and we have the following commutative diagram:

\[
\xymatrix{
X_{1}\times X_{2} \ar[r]^{f_{1}\times f_{2}} \ar[d]_{\pi} & X_{1}\times X_{2} \ar[d]^{\pi}\\
X \ar[r]_{f} & X.
}
\]
Here $\pi$ is the isogeny defined by $\pi (x_{1},x_{2})=x_{1}+x_{2}$.

Since $X$ is divisible, we have $\End(X) \subset \End(X) {\otimes}_{\Z}\Q$.
Let $f \in \End(X)$ and $F(t) \in \Z[t]$ be the monic minimal polynomial of $f$ as an element of $\End(X) {\otimes}_{\Z}\Q$.
(The monic minimal polynomial has integer coefficients because those of endomorphisms of a torus and an abelian variety
have integer coefficients.)
Let
\[
F(t)=F_{0}(t)^{e_{0}}F_{1}(t)^{e_{1}}\cdots F_{r}(t)^{e_{r}}
\]
be the decomposition into irreducible factors where $F_{0}(t)=t$, $e_{0}\geq0$, $e_{i}>0, i=1,\dots,r$
and $F_{i}(t)$ are distinct monic irreducible polynomials.
Note that $r$ is possibly zero.
Set
\[
X_{i}=F_{0}(f)^{e_{0}} \cdots F_{i-1}(f)^{e_{i-1}} F_{i+1}(f)^{e_{i+1}}\cdots F_{r}(f)^{e_{r}}(X)
\]
and $f_{i}=f|_{X_{i}}$.
Here, $X_{i}$ are also (semi-)abelian varieties since they are images of a (semi-)abelian variety.
Then we get the commutative diagram
\[
\xymatrix@C=46pt{
X_{0}\times \cdots \times X_{r} \ar[r]^{f_{0}\times \cdots \times f_{r}} \ar[d]_{\pi}&X_{0}\times \cdots \times X_{r} \ar[d]^{\pi}\\
X \ar[r]_{f}& X
}
\]
where $\pi(x_{0},\dots,x_{r})=x_{0}+\cdots+x_{r}$.
Note that the monic minimal polynomial of $f_{i}$ as an element of $\End(X_{i}) {\otimes}_{\Z}\Q$ is $F_{i}(t)^{e_{i}}$.
Note that $f$ is surjective if and only if $e_{0}=0$ and if this is the case, we have
$\delta_{f}=\delta_{f_{0}\times \cdots \times f_{r}}=\max\{ \delta_{f_{1}},\dots, \delta_{f_{r}}\}$ (cf.  Remark \ref{rmk on dyn deg}).

\section{Arithmetic and dynamical degrees of isogenies of abelian varieties}\label{sec:abelian varieties}

 \begin{thm}[Theorem \ref{main semiab}(3) for abelian varieties]\label{main}
 Let $X$ be an abelian variety and $f \colon X \longrightarrow X$ be a group homomorphism, both defined over $ \overline{\mathbb Q}$. 
 Let $F(t)$ be the monic minimal polynomial of $f$ as an element of $\End(X)_{\Q}$ and
\[
F(t)=t^{e_{0}}F_{1}(t)^{e_{1}}\cdots F_{r}(t)^{e_{r}}
\]
the irreducible decomposition in $\Q[t]$ where $e_{0}\geq0$ and  $e_{i}>0$ for $i=1,\dots, r$. 
Then we have
\[
A(f)= \{1, \rho(F_{1})^{2},\dots, \rho(F_{r})^{2} \}
\]
\end{thm}

\begin{thm}[Theorem \ref{main2 semiab} for isogenies of abelian varieties]\label{main2}
Let $X$ be an abelian variety and $f \colon X \longrightarrow X$ a surjective group homomorphism, both defined over $ \overline{\mathbb Q}$.
Suppose that the minimal polynomial of $f$ has no irreducible factor that is
a cyclotomic polynomial.
Then for any $x \in X(\QQ)$,
\[
\alpha_{f}(x)=1 \iff \#\ O_{f}(x)<\infty \iff x\in X( \overline{\mathbb Q})_{\rm tors}
\]
where $X( \overline{\mathbb Q})_{\rm tors}$ is the set of torsion points.
\end{thm}

\begin{lem}\label{dyn deg of isog}
Let $X$ be an abelian variety of dimension $g$ over an algebraically closed field of characteristic zero
and $f \colon X \longrightarrow X$ a surjective group homomorphism.
Let $P(t)$ be the monic minimal polynomial of $f$ as an element of $\End(X)_{\Q}$, which has integer coefficient, and
$\rho$ the maximum among the absolute values of the roots of $P(t)$.
Then we have $\delta_{f}=\rho^{2}$.
\end{lem}

\begin{rmk}
The minimal polynomial of $f$ as an element of $\End(X)_{\Q}$
is equal to the minimal polynomial of $T_{l}(f)$ for every prime number $l$.
If the ground field is $\C$, these are also equal to the minimal polynomial of the analytic representation of $f$.
\end{rmk}

\begin{proof}
By the Lefschetz principle, we may assume that the ground field is $\C$.
Let $X=\C^{g}/\Lambda$, where $\Lambda$ is a lattice in $\C^{g}$.
Let $f_{r} \colon \Lambda \longrightarrow \Lambda$ be the rational representation 
and $f_{a} \colon \C^{g} \longrightarrow \C^{g}$ the analytic representation
of $f$.

We have a natural isomorphism $H^{r}(X;\Z) \simeq \Hom_{\Z}(\bigwedge^{r}\Lambda,\Z)$ (cf.  \cite[\S 1 (4)]{mum}).
If we identify $H^{r}(X;\Z)$ with $\Hom_{\Z}(\bigwedge^{r}\Lambda,\Z)$ by this isomorphism,
then $f^{*} \colon H^{r}(X;\Z) \longrightarrow H^{r}(X;\Z)$ is $(\bigwedge^{r}f_{r})^{*}$.
Therefore, the eigenvalues of $f^{*}$ are products of $r$ eigenvalues of $f_{r}$.
Since $f_{a}|_{\Lambda}=f_{r}$, the characteristic polynomial of $f_{r}$ as an $\R$-linear map is $Q(t)\overline{Q(t)}$ where $Q(t)$ is the
characteristic polynomial of $f_{a}$ as a $\C$-linear map. 
(Take a basis $e_{1},\dots, e_{g}$ of $\C^{g}$ so that $f_{a}$ is represented by an upper triangular matrix.
Then compute the characteristic polynomial of $f_{a}, f_{r}$ using bases $\{e_{1} ,\dots, e_{g}\}, \{e_{1}, ie_{1},\dots,e_{g}, ie_{g}\}$ respectively.)
Note that the set of roots of $P(t)$ and $Q(t)$ are the same.
Therefore, the spectral radius of $f^{*} \colon H^{2}(X;\Z) {\otimes}_{\Z}\R \longrightarrow H^{2}(X; \Z) {\otimes}_{\Z}\R$ is equal to
the square of spectral radius of $f_{r}$.
Note that the spectral radius of $f^{*}\curvearrowright H^{2}(X; \Z)$
is equal to the spectral radius of $f^{*}\curvearrowright H^{1,1}(X)$ (cf.  the inequality above Proposition 4.4 in \cite{ds}),
this proves the theorem.
\end{proof}

Now,  let $X$ be an abelian variety and $f \colon X \longrightarrow X$ a group homomorphism, both defined over $ \overline{\mathbb Q}$. 
Let $F(t)$ be the monic minimal polynomial of $f$ and
\[
F(t)=t^{e_{0}}F_{1}(t)^{e_{1}}\cdots F_{r}(t)^{e_{r}}
\]
the decomposition into irreducible factors in $\Q[t]$. 
Here $F_{i}$ are distinct monic irreducible polynomial in $\Z[t]$ with $F_i(0)\neq0$.
Write $F_{0}(t)=t$.
Set
\[
X_{i}=F_{0}(f)^{e_{0}} \cdots F_{i-1}(f)^{e_{i-1}} F_{i+1}(f)^{e_{i+1}}\cdots F_{r}(f)^{e_{r}}(X).
\]
Then by \S3, we have the following commutative diagram:
\[
\xymatrix@C=46pt{
X_{0}\times \cdots \times X_{r} \ar[r]^{f_{0}\times \cdots \times f_{r}} \ar[d]_{\pi} &X_{0}\times \cdots \times X_{r} \ar[d]^{\pi}\\
X \ar[r]_{f}& X.
}
\]
Here, the vertical arrows are isogenies.
Note that the minimal polynomial of $f_{i}$ is $F_{i}(t)^{e_{i}}$.

\begin{lem}\label{power of irr case}
Let $f \colon X \longrightarrow X$ be a surjective group homomorphism over $  \overline{\mathbb Q}$ such that the minimal polynomial of $f$
is the form of $F(t)^{e}$ where $F$ is an irreducible monic polynomial in $\Z[t]$.
For any $x\in X(\QQ)$, if $ \alpha_{f}(x)<\delta_{f}$, then $x$ is a torsion point.
In particular, $x$ is an $f$-preperiodic point
and $ \alpha_{f}(x)=1$.
\end{lem}

\begin{rmk}
Note that $ \alpha_{f}(x)<\delta_{f}$ happens only if $\delta_{f}>1$.
In the above situation, $\delta_{f}=1$ if and only if $F(t)$ is a cyclotomic polynomial.
This follows from Lemma \ref{dyn deg of isog} and the fact that if the absolute value of every root of an irreducible monic polynomial
with integer coefficients is less than or equal to one, then the polynomial is cyclotomic.
\end{rmk}

\begin{proof}
We prove the claim by induction on $\dim X$.
If $\dim X =0$, there is nothing to prove.
Suppose $\dim X=d>0$ and the claim holds for dimension $<d$.
Take a nef $\R$-divisor $D$ such that $f^{*}D\equiv \delta_{f}D$.
Let $q$ be the quadratic part of the canonical height of $D$,
i.e. $q(x)=\lim_{n\to \infty} h_{D}(nx)/n^{2}$.
By \cite[Theorem 29, Lemma 31]{ks1}, there exists an $f$-invariant subabelian variety $Y \subsetneq X$ such that
\[
\{x\in X( \overline{\mathbb Q}) \mid q(x)=0\}=Y( \overline{\mathbb Q}) + X( \overline{\mathbb Q})_{\rm tors}.
\]
Assume $ \alpha_{f}(x)<\delta_{f}$.
Then $x=y+z$ for some $y\in Y(\QQ)$ and some torsion point $z$.
It is enough to show that $y$ is a torsion point.
If $Y$ is a point, we are done.
Suppose $\dim Y>0$.
Since $Y$ is $f$-invariant, the minimal polynomial of $f|_{Y}$ divides $F(t)^{e}$ and is not equal to $1$.
Thus $\delta_{f|_{Y}}=\delta_{f}> \alpha_{f}(x)= \alpha_{f}(y)= \alpha_{f|_{Y}}(y)$.
Here, we use the fact that $ \alpha_{f}(x)= \alpha_{f}(y+z)= \alpha_{f}(y)$.
This follows from the definition of arithmetic degree and the fact that
the Neron-Tate height associated with a symmetric ample divisor is invariant under the translation by a torsion point.
By the induction hypothesis, $y$ is a torsion point.
\end{proof}

\begin{proof}[Proof of Theorem \ref{main}]

We use the notation of \S3.
Set $f_{i}=f|_{X_{i}}$.
By  \cite[Lemma 6]{sil2}, $A(f)=A(f_{0}\times \cdots \times f_{r})$.
Since $ \alpha_{f_{i}}(0)=1$ and $ \alpha_{f_{0}\times \cdots \times f_{r}}(x_{0},\dots ,x_{r})=\max\{ \alpha_{f_{0}}(x_{0}), \dots, \alpha_{f_{r}}(x_{r})\}$
(see Lemma \ref{product case}), we have $A(f_{0}\times \cdots \times f_{r})=A(f_{0})\cup \cdots \cup A(f_{r})$.
Note that $A(f_{0})=\{1\}$ since $f_{0}^{e_{0}}=0$. 
By Lemma \ref{power of irr case} and the fact that there always exists a point whose arithmetic degree equals the dynamical degree (cf.  \cite[Corollary 32]{ks1} or \cite[Theorem 1.6]{mss}),
we have $A(f_{i})=\{1, \delta_{f_{i}}\}$ for $i=1,\dots, r$.
Thus $A(f)=\{1, \delta_{f_{1}}, \dots, \delta_{f_{r}}\}$.
By Lemma \ref{dyn deg of isog}, $\delta_{f_{i}}$ is equal to $\rho(F_{i})^{2}$.
\end{proof}

\begin{proof}[Proof of Theorem \ref{main2}]
By \S3, we may assume the minimal polynomial of $f$ is the form of 
$F(t)^{e}$ where $F$ is an irreducible polynomial that is not cyclotomic.
Then  $\rho(F)$ is greater than one.
Thus $\delta_{f}>1$.
By Lemma \ref{power of irr case}, if $ \alpha_{f}(x)=1$ then $x$ is a torsion point.
\end{proof}

\section{Arithmetic and dynamical degrees of self-morphisms of semi-abelian varieties}

\subsection{Dynamical degrees}\label{subsec:semi-abelian dyn deg}

In this subsection, the ground field is an algebraically closed field of characteristic zero.

\begin{prop}\label{dyn deg of isog of semiab}
Let $X$ be a semi-abelian variety.
Let $f \colon X \longrightarrow X$ be a surjective group homomorphism.
Let 
\[
\xymatrix{
0 \ar[r] & T \ar[r] & X \ar[r]^{\pi} & A \ar[r] &0
}
\]
 be an exact sequence where $T$ is a torus and $A$ is an abelian variety.
 Then $f$ induces surjective group homomorphisms
 \begin{align*}
 f_{T}:=f|_{T}  & \colon T \longrightarrow T\\
 g & \colon  A \longrightarrow A
 \end{align*}
 with $g\circ \pi=\pi \circ f$.
 Then we have
 \[
 \delta_{f}=\max\{\delta_{g}, \delta_{f_{T}}\}
 \]
\end{prop}

\begin{rmk}
This follows from the product formula of dynamical degrees (\cite[Theorem 1.1]{dn}) and \cite[Remark 3.4]{dn}. 
To apply \cite[Remark 3.4]{dn}, just take the standard compactification of $X$ as in \cite[\S 2 (2.3)]{vo}.
The proofs of \cite[Theorem 1.1]{dn} and \cite[Remark 3.4]{dn} are based on analytic methods, 
so we give an algebraic proof of Proposition \ref{dyn deg of isog of semiab} below.
\end{rmk}

\begin{lem}\label{pull-backs}
Let $\xymatrix{X \ar@{-->}[r]^{f}  & Y \ar@{-->}[r]^{g} & Z}$ be rational maps of smooth projective varieties.
Suppose $f(X\setminus I_{f}) \not\subset I_{g}$ where $I_{f}, I_{g}$ are the indeterminacy loci of $f, g$.
Then for any free divisor $H$ on $Z$, we have
\[
(g\circ f)^{*}H \leq f^{*}(g^{*}H).
\]
Here, for divisor classes $A$ and $B$, $A\leq B$ means $B-A$ is represented by an effective divisor.
\end{lem}
\begin{proof}
Take resolutions $\pi_{X}, \pi_{Y}$ as follows:
\[
\xymatrix{
\widetilde{X} \ar[dd]_{\pi_{X}} \ar[rd]^{ \widetilde{f}} & & \\
& \widetilde{Y} \ar[d]_{\pi_{Y}} \ar[rd]^{ \widetilde{g}} &\\
X \ar@{-->}[r]_{f} & Y \ar@{-->}[r]_{g} & Z
}
\]
where $ \widetilde{X}, \widetilde{Y}$ are smooth projective varieties
and $\pi_{Y} \colon \pi_{Y}^{-1}(Y\setminus I_{g}) \simeq Y\setminus I_{g}$.
Then
\begin{align*}
(g\circ f)^{*}H & = {\pi_{X}}_{*} \widetilde{f}^{*} \widetilde{g}^{*}H\\
f^{*}(g^{*}H) & =  {\pi_{X}}_{*} \widetilde{f}^{*} \pi_{Y}^{*} {\pi_{Y}}_{*} \widetilde{g}^{*}H.
\end{align*}
Since $ \widetilde{g}^{*}H$ is free, the divisor $\pi_{Y}^{*} {\pi_{Y}}_{*} \widetilde{g}^{*}H-\widetilde{g}^{*}H$
is represented by an effective divisor with support contained in the exceptional locus $\Exc(\pi_{Y})$ of $\pi_{Y}$. 
Since $ \widetilde{f}( \widetilde{X}) \not\subset \Exc(\pi_{Y})$, we have 
${\pi_{X}}_{*} \widetilde{f}^{*} (\pi_{Y}^{*} {\pi_{Y}}_{*} \widetilde{g}^{*}H-\widetilde{g}^{*}H)\geq 0$.
\end{proof}

\begin{proof}[Proof of Proposition \ref{dyn deg of isog of semiab}]\ 
We will write the multiplication of the groups $X, A, T$ by addition.
Take a non-empty open subset $U \subset A$ and a section $s \colon U \longrightarrow \pi^{-1}(U)$ of $\pi$.
(There exists such a section because of the structure theorem of semi-abelian varieties \cite[Lemma 2.2]{vo}.)
Then
\[
\xymatrix@C=36pt@R=5pt{
\pi^{-1}(U) \ar[r]^{\sim} & U \times T\\
P  \ar@{}[u]|{\rotatebox[origin=c]{90}{$\in$}} \ar@{|->}[r] & (\pi(P), P-s(\pi(P))) \ar@{}[u]|{\rotatebox[origin=c]{90}{$\in$}}.
}
\]
By this isomorphism, $f$ is conjugate to the rational map
\[
\xymatrix@C=36pt@R=5pt{
A\times T \ar@{-->}[r]^{ \widetilde{f}} & A \times T\\
(x,y)  \ar@{}[u]|{\rotatebox[origin=c]{90}{$\in$}} \ar@{|->}[r] & (g(x), f_{T}(y)+h(x)) \ar@{}[u]|{\rotatebox[origin=c]{90}{$\in$}}
}
\]
where $h(x)=f(s(x))-s(g(x))$.
Note that $h$ is defined on $V:=U\cap g^{-1}(U)$ and $h(V)\subset T$.
Fix a compactification $T \subset \overline{T}=\P^{1}\times \cdots \times \P^{1}$.
The rational map $A\times \overline{T} \dashrightarrow A \times \overline{T}$ defined by $ \widetilde{f}$ is 
also denoted by $ \widetilde{f}$.

\begin{claim}\label{pull-back mult}
Let $m \colon \overline{T}\times \overline{T} \dashrightarrow \overline{T}$ be the rational map
defined by the multiplication morphism $T \times T \longrightarrow T$.
Then, for any divisor $D$ on $ \overline{T}$, $m^{*}D\sim\pr_{1}^{*}D+\pr_{2}^{*}D$. 
\end{claim}
\begin{claimproof}
We can write $m^{*}D\sim\pr_{1}^{*}D_{1}+\pr_{2}^{*}D_{2}$ where $D_{1}, D_{2}$ are divisors on $ \overline{T}$.
Let $1\in T$ be the neutral element.
Let $i \colon \overline{T} \longrightarrow \overline{T}\times \overline{T}$ be the map defined by
$i(t)=(t,1)$.
Then, $i( \overline{T})\cap I_{m}= \emptyset$.
Therefore, we have $D_{1}\sim i^{*}m^{*}D=(m\circ i)^{*}D=D$.
In the same way, we can show that $D_{2}\sim D$.
\end{claimproof}

Since $ \overline{T}$ is a product of $\P^{1}$, we have 
$\CH^{1}(A \times \overline{T})=\pr_{1}^{*}\CH^{1}(A)\oplus \pr_{2}^{*}\CH^{1}( \overline{T})\simeq 
\CH^{1}(A)\oplus \CH^{1}( \overline{T})$.

\begin{claim}\label{pull-back matrix}
We have
\[
\widetilde{f}^{*}= 
\begin{pmatrix}
g^{*}& h^{*}\\
0&f_{T}^{*}
\end{pmatrix}
\colon \CH^{1}(A)\oplus \CH^{1}( \overline{T}) \longrightarrow \CH^{1}(A)\oplus \CH^{1}( \overline{T}).
\]
\end{claim}
\begin{claimproof}
It is enough to prove the following two statements:
\begin{enumerate}
\item[(1)] $ \widetilde{f}^{*}\pr_{1}^{*}H_{A} = \pr_{1}^{*}g^{*}H_{A}$ for every ample divisor $H_{A}$ on $A$.
\item[(2)] $ \widetilde{f}^{*}\pr_{1}^{*}H_{ \overline{T}} = \pr_{1}^{*}h^{*}H_{ \overline{T}}+ \pr_{2}^{*}f_{T}^{*}H_{ \overline{T}}$
in $\CH^{1}(A \times \overline{T})$
for every very ample divisor $H_{ \overline{T}}$ on $ \overline{T}$.
\end{enumerate}

(1)
Since $\pr_{1}$ is a morphism, we have
\begin{align*}
\widetilde{f}^{*}\pr_{1}^{*}H_{A}=(\pr_{1}\circ \widetilde{f})^{*}H_{A}
=(g\circ \pr_{1})^{*}H_{A}=\pr_{1}^{*}g^{*}H_{A}.
\end{align*}

(2)
First, in $\CH^{1}(A \times \overline{T})$, we have
\begin{align*}
\widetilde{f}^{*}\pr_{2}^{*}H_{ \overline{T}}&=(\pr_{2}\circ \widetilde{f})^{*}H_{ \overline{T}}\\
&=(m\circ (h\times f_{T}))^{*}H_{ \overline{T}} \\
&\leq (h\times f_{T})^{*}m^{*}H_{ \overline{T}}  &  \qquad \text{by Lemma \ref{pull-backs}} \\
&= (h\times f_{T})^{*}(\pr_{1}^{*}H_{ \overline{T}}+\pr_{2}^{*}H_{ \overline{T}} )
 	&\qquad \text{by Claim \ref{pull-back mult}} \\
&=\pr_{1}^{*}h^{*}H_{ \overline{T}}+\pr_{2}^{*}f_{T}^{*}H_{ \overline{T}}.
\end{align*}
Now take an effective divisor $E$ on $A \times \overline{T}$ that represents
the class $(h\times f_{T})^{*}m^{*}H_{ \overline{T}}-(m\circ (h\times f_{T}))^{*}H_{ \overline{T}}$.
For a general closed point $a \in A$, $\Supp E$ does not contain $\{a\} \times \overline{T}$.
Let $i_{a} \colon \overline{T}=\{a\} \times \overline{T} \subset A \times \overline{T}$ be the inclusion.
Since $i_{a}^{*}E$ is effective, we have
\begin{align*}
i_{a}^{*}( \widetilde{f}^{*}\pr_{2}^{*}H_{ \overline{T}})\leq i_{a}^{*}(\pr_{1}^{*}h^{*}H_{ \overline{T}}+\pr_{2}^{*}f_{T}^{*}H_{ \overline{T}})
=f_{T}^{*}H_{ \overline{T}}.
\end{align*}
Similarly, if $b\in \overline{T}$ is a general closed point and $j_{b} \colon A=A\times \{b\}\subset A \times \overline{T}$
is the inclusion, we have
\begin{align*}
j_{b}^{*}( \widetilde{f}^{*}\pr_{2}^{*}H_{ \overline{T}})\leq j_{b}^{*}(\pr_{1}^{*}h^{*}H_{ \overline{T}}+\pr_{2}^{*}f_{T}^{*}H_{ \overline{T}})
=h^{*}H_{ \overline{T}}.
\end{align*}
Therefore, if we write $ \widetilde{f}^{*}\pr_{2}^{*}H_{ \overline{T}}=\pr_{1}^{*}D_{1}+\pr_{2}^{*}D_{2}$
where $D_{1}$ and $D_{2}$ are divisor classes on $A$ and $ \overline{T}$ respectively,
we have proved 
\begin{align*}
D_{1}\leq h^{*}H_{ \overline{T}},\ D_{2}\leq f_{T}^{*}H_{ \overline{T}}.
\end{align*}

On the other hand, for a general $a\in V\subset A$, $\{a\}\times \overline{T}$ is not contained in the indeterminacy locus of $ \widetilde{f}$,
and $\pr_{2}\circ \widetilde{f} \circ i_{a}= T_{h(a)}\circ f_{T}$.
Here, the translation $T_{h(a)}$ defines an automorphism on $ \overline{T}$ and induces identity on $\CH^{1}( \overline{T})$.
Thus
\begin{align*}
f_{T}^{*}H_{ \overline{T}}
&=f_{T}^{*}T_{h(a)}^{*}H_{ \overline{T}}=(T_{h(a)}\circ f_{T})^{*}H_{ \overline{T}}\\
&=(\pr_{2}\circ \widetilde{f} \circ i_{a})^{*}H_{ \overline{T}}=(\widetilde{f} \circ i_{a})^{*}\pr_{2}^{*}H_{ \overline{T}}\\
&\leq i_{a}^{*} \widetilde{f}^{*} \pr_{2}^{*} H_{ \overline{T}} & \text{by Lemma \ref{pull-backs}}\\
&=i_{a}^{*}(\pr_{1}^{*}D_{1}+\pr_{2}^{*}D_{2})=D_{2}.
\end{align*}
Hence we get $D_{2}\leq f_{T}^{*}H_{ \overline{T}}\leq D_{2}$ and therefore $D_{2}= f_{T}^{*}H_{ \overline{T}}$.

Similarly, since $A\times \{1\}$ is not contained in the indeterminacy locus of $ \widetilde{f}$,
we have
\begin{align*}
h^{*}H_{ \overline{T}}&=(\pr_{2}\circ \widetilde{f}\circ j_{1})^{*}H_{ \overline{T}}
=( \widetilde{f}\circ j_{1})^{*}\pr_{2}^{*}H_{ \overline{T}}\\
&\leq j_{1}^{*} \widetilde{f}^{*} \pr_{2}^{*} H_{ \overline{T}} & \text{by Lemma \ref{pull-backs}}\\
&=j_{1}^{*}(\pr_{1}^{*}D_{1}+\pr_{2}^{*}D_{2})=D_{1}.
\end{align*}
Thus $D_{1}=h^{*}H_{ \overline{T}}$.

\end{claimproof}

Note that $ \widetilde{(f^{n})}= \widetilde{f}^{n}$.
Set $h_{n}=f^{n}\circ s-s\circ g^{n}$.
Then we have
\begin{align*}
(\widetilde{f}^{n})^{*}=
\begin{pmatrix}
(g^{n})^{*}& h_{n}^{*}\\
0&(f_{T}^{n})^{*}
\end{pmatrix}.
\end{align*}
Note also that $N^{1}(A \times \overline{T})=(\CH^{1}(A)/\equiv) \oplus \CH^{1}( \overline{T})$
and the action of $(\widetilde{f}^{n})^{*}$ on $N^{1}(A \times \overline{T})$ is in the same form.
By \cite[Theorem 15]{ks3}, $\delta_{ \widetilde{f}}=\lim_{n\to \infty }\rho(( \widetilde{f}^{n})^{*}|_{N^{1}(A \times \overline{T})})^{1/n}$.
Thus
\begin{align*}
\delta_{f}=\delta_{ \widetilde{f}}=\lim_{n\to \infty }\max\{\rho((g^{n})^{*}|_{N^{1}(A)}), \rho((f_{T}^{n})^{*}|_{N^{1}( \overline{T})})\}^{1/n}
=\max\{\delta_{g},\delta_{f_{T}}\}.
\end{align*}

\end{proof}

\begin{lem}\label{translation and dynamical degree}
Let $f \colon X \longrightarrow X$ be a surjective group homomorphism of a semi-abelian variety $X$ and
$a \in X$ a closed point. Then $\delta_{T_{a}\circ f}=\delta_{f}$.
\end{lem}
\begin{proof}
Let $ \overline{X}$ be the standard compactification of $X$ as in \cite[\S 2 (2.3)]{vo}.
Then $T_{a}$ extends to an automorphism of $ \overline{X}$, which we also denote by $ T_{a}$, and
the pull-back $ T_{a}^{*} \colon N^{1}( \overline{X}) \longrightarrow N^{1}( \overline{X})$ is the identity.
(We can deduce these facts from the description of the group law in terms of the compactification, cf.  \cite[the proof of Proposition 2.6]{vo}.)
Thus, as an endomorphisms of $N^{1}( \overline{X})$, we have
\[
((T_{a}\circ f)^{n})^{*}=(T_{b}\circ f^{n})^{*}=(f^{n})^{*}\circ T_{b}^{*}=(f^{n})^{*}
\]
where $b=a+f(a)+\cdots +f^{n-1}(a)$.
Therefore, 
\[
\delta_{T_{a}\circ f}=\lim_{n \to \infty}\left\|((T_{a}\circ f)^{n})^{*}\right\|^{1/n}=\lim_{n \to \infty}\left\|(f^{n})^{*}\right\|^{1/n}=\delta_{f}
\]
where $\left\| \cdot \right\|$ is a norm on $\End_{\R}(N^{1}( \overline{X})_{\R})$.
\end{proof}

\begin{proof}[Proof of Theorem \ref{main dyn deg}]
(2) is Lemma \ref{translation and dynamical degree}.
(1) follows from Proposition \ref{dyn deg of isog of semiab}, Lemma \ref{dyn deg of isog} and Remark \ref{dyn deg of monomial map}.
\end{proof}

\subsection{Kawaguchi-Silverman conjecture and arithmetic degrees}\label{subsec:semi-abelian arith deg}
 
 In this subsection, the ground field is $ \overline{\mathbb Q}$.

\begin{lem}\label{core case}
Let $f \colon X \longrightarrow X$ be a surjective group homomorphism of a semi-abelian variety.
Fix an exact sequence
\[
\xymatrix{
0 \ar[r] & T \ar[r] & X \ar[r]^{\pi} & A \ar[r] &0.
}
\]
The morphisms induced by $f$ is denoted by
\begin{align*}
f_{T} \colon  &T \longrightarrow T\\
g \colon  & A \longrightarrow A.
\end{align*}
Suppose the minimal polynomial of $f$ as an element of  $\End(X) {\otimes}_{\Z}\Q$ is the form of $F(t)^{e}$
where $F(t)$ is a monic irreducible polynomial that is not cyclotomic and $e>0$.
(Note that the minimal polynomial is automatically monic with integer coefficient because it is the case for $f_{T}$ and $g$.)
Then, for $x\in X(\QQ)$, either
\begin{enumerate}
\item[\rm (1)] $O_{g}(\pi(x))$ is infinite and $ \alpha_{f}(x)=\delta_{f}$ or,
\item[\rm (2)] $\pi(x)$ is a torsion point and $ \alpha_{f}(x)=1 \text{ or } \delta_{f_{T}} $.
\end{enumerate}
Moreover, 
\[
A(f)=\{1, \delta_{f_{T}}, \delta_{g}\}=
\begin{cases}
\{1, \rho(F)\} \qquad \text{if $X=T$,}\\
\{1, \rho(F)^{2}\} \qquad \text{if $X=A$,}\\
\{1, \rho(F), \rho(F)^{2}\} \qquad \text{otherwise}.
\end{cases}
\]

\end{lem}

\begin{lem}\label{multiple by n}
Let $X$ be a semi-abelian variety and $f \colon X \longrightarrow X$ be a surjective group homomorphism.
Let $x \in X(\QQ)$ be a point and $n>0$ a positive integer.
If $ \alpha_{f}(nx)$ exists, then $ \alpha_{f}(x)$ also exists and $ \alpha_{f}(nx)= \alpha_{f}(x)$.
\end{lem}
\begin{proof}
Let $X''$ and $X'$ be smooth projectivization of $X$ such that the multiplication morphism $[n] \colon X \longrightarrow X$
becomes a morphism $\pi \colon X'' \longrightarrow X'$.
Let $f'' \colon X'' \longrightarrow X''$ and $f' \colon X' \longrightarrow X'$ be the dominant rational maps induced by $f$.
Since $f$ is a group homomorphism, we have $f' \circ \pi=\pi \circ f''$.
Moreover, we have $\pi^{-1}(X)=X$ since $[n] \colon X \longrightarrow X$ is finite.
By Lemma \ref{comparing arithmetic deg}, we get the assertion.
\end{proof}

\begin{proof}[Proof of Lemma \ref{core case}]
First of all, we have
\[
\delta_{f}=\max\{\delta_{g}, \delta_{f_{T}} \}=\max\{ \rho(F)^{2}, \rho(F)\}=\rho(F)^{2}=\delta_{g}
\]
(see Theorem \ref{key known thm}(2), Proposition \ref{dyn deg of isog of semiab}, Lemma \ref{dyn deg of isog}).  
By Lemma \ref{power of irr case}, we have
\begin{align*}
\alpha_{g}(\pi(x))=
\begin{cases}
1 \qquad \text{if $\pi(x)$ is torsion}\\
\delta_{g}=\delta_{f} \qquad \text{otherwise.}
\end{cases}
\end{align*}
Note that, by Lemma \ref{ascent} and Remark \ref{rmk on dyn deg} (3), we have $ \alpha_{g}(\pi(x))\leq \underline{\alpha}_{f}(x)\leq \delta_{f}$. 
Thus, if $ \alpha_{g}(\pi(x))=\delta_{f}$, $ \alpha_{f}(x)$ exists and is equal to $\delta_{f}$.

Now, suppose $\pi(x)$ is a torsion point.
Take a positive integer $n$ such that $n\pi(x)=0$.
Then $nx \in T$ and therefore $ \alpha_{f}(nx)= \alpha_{f_{T}}(nx)$ exists and is equal to $1$ or $\rho(F)=\delta_{f_{T}}$ (Theorem \ref{key known thm}(2)).
By Lemma \ref{multiple by n},  $ \alpha_{f}(x)$ exists and is equal to $1$ or $\delta_{f_{T}}$.

The claim $A(f)=\{1, \delta_{f_{T}}, \delta_{f}\}$  follows from the facts that $A(f_{T})=\{1, \delta_{f_{T}}\}$ (Theorem \ref{key known thm}(2)),
$A(g)=\{1, \delta_{g}\}$ (Lemma \ref{power of irr case}) and
$ \alpha_{f}(x)\geq \alpha_{g}(\pi(x))$ (Lemma \ref{ascent}).  
\end{proof}

\begin{lem}\label{minimal polynomial and conjugation}
Let $f \colon X \longrightarrow X$ be a group homomorphism of a semi-abelian variety.
Let $F(t)$ be the minimal monic polynomial of $f$.
Assume $F(1)\neq 0$.
Let $a\in X(\QQ)$ be any point.
Then there exists a point $b\in X(\QQ)$ such that $h:=T_{b}\circ (T_{a}\circ f) \circ T_{-b}$ is a group homomorphism.
For every such $b$, the minimal polynomial of $h$ is also $F(t)$.
\end{lem}
\begin{proof}
Since $F(1)\neq 0$, $f-\id$ is surjective.
For any $b\in X(\QQ)$ with $f(b)-b=a$, the morphism $T_{b}\circ (T_{a}\circ f) \circ T_{-b}$ is a group homomorphism.

Now we prove the second part.
By symmetry, it is enough to prove $F(h)=0$.
We have
\[
h^{n}=T_{b}\circ (T_{a}\circ f)^{n}\circ T_{-b}=T_{b}\circ T_{a+f(a)+\cdots +f^{n-1}(a)} \circ f^{n}\circ T_{-b}.
\]
Note that since $h$ is a group homomorphism, we have $h(0)=0$, in other words, $a=(f-\id)(b)$.
Thus 
\[
h^{n}=T_{b}\circ T_{f^{n}(b)-b}\circ f^{n}\circ T_{-b}=T_{f^{n}(b)}\circ f^{n} \circ T_{-b}.
\]
Therefore, for any $x\in X(\QQ)$
\[
F(h)(x)=F(f)(b)+F(f)(x-b)=0.
\]
\end{proof}

\begin{proof}[Proof of Theorem \ref{main semiab}]

Let $X$ be a semi-abelian variety and first assume $f \colon X \longrightarrow X$ is a group homomorphism.
We use the notation of \S \ref{splitting}.
Apply Lemma \ref{comparing arithmetic deg} for a suitable smooth compactification of
\[
\xymatrix@C=46pt{
X_{0}\times \cdots \times X_{r} \ar[r]^{f_{0}\times \cdots \times f_{r}} \ar[d]_{\pi}&X_{0}\times \cdots \times X_{r} \ar[d]^{\pi}\\
X \ar[r]_{f}& X.
}
\]
By  Lemma \ref{core case}, $ \alpha_{f_{i}}(x)$ exists for every $i$ and every point $x \in X_{i}(\QQ)$.
Therefore, by Lemma \ref{product case} and Lemma \ref{comparing arithmetic deg},  $A(f)=A(f_{0}\times \cdots \times f_{r})=A(f_{0})\cup \cdots \cup A(f_{r})$.
Since $f_{0}$ is nilpotent, $A(f_{0})=\{1\}$.
If $F_{i}$ is a cyclotomic polynomial, then $\delta_{f_{i}}=1$ and $A(f_{i})=\{1\}$.
Therefore by Lemma \ref{core case}, we have
\begin{align*}
A(f)&=A(f_{1})\cup \cdots \cup A(f_{r})\\
&=\{1\}\cup A_{1}\cup \cdots \cup A_{r}.
\end{align*}

Now,  consider any self-morphism of $X$.
Any self-morphism is the form of $T_{a}\circ f$ where $T_{a}$ is the translation by $a\in X(\QQ)$ and $f$ is a group homomorphism.
There exist points $a_{i}\in X_{i}(\QQ)$ such that $\pi(a_{0},\dots, a_{r})=a_{0}+\cdots +a_{r}=a$.
Then we have the following commutative diagram:
\[
\xymatrix@C=46pt{
X_{0}\times \cdots \times X_{r} \ar[r]^{f_{0}\times \cdots \times f_{r}} \ar[d]_{\pi}&X_{0}\times \cdots \times X_{r} \ar[d]^{\pi} \ar[r]^{T_{a_{0}}\times \cdots \times T_{a_{r}}} & X_{0}\times \cdots \times X_{r} \ar[d]^{\pi}\\
X \ar[r]_{f}& X \ar[r]_{T_{a}} & X.
}
\]
As above, we have $A(T_{a}\circ f)=A\bigl((T_{a_{0}}\circ f_{0}) \times \cdots \times (T_{a_{r}}\circ f_{r})\bigr)$.
Since $f_{0}$ is nilpotent, every orbit of $T_{a_{0}}\circ f_{0}$ is finite and therefore $A(T_{a_{0}}\circ f_{0})=\{1\}=A(f_{0})$.
If $F_{i}(t)$ is a cyclotomic polynomial, by Lemma \ref{translation and dynamical degree} we have $\delta_{T_{a_{i}}\circ f_{i}}=\delta_{f_{i}}=1$
and therefore $A(T_{a_{i}}\circ f_{i})=\{1\}=A(f_{i})$.
If $F_{i}(t), i\geq1$ is not a cyclotomic polynomial,  by Lemma \ref{minimal polynomial and conjugation},
$T_{a_{i}}\circ f_{i}$ is conjugate by a translation to a group homomorphism $h_{i}$ with minimal polynomial $F_{i}^{e_{i}}$.
In particular, $A(T_{a_{i}}\circ f_{i})=A(h_{i})=A(f_{i})$.
Therefore 
\begin{align*}
A\bigl((T_{a_{0}}\circ f_{0}) \times \cdots \times (T_{a_{r}}\circ f_{r})\bigr)&=A(T_{a_{0}}\circ f_{0})\cup \cdots \cup A(T_{a_{r}}\circ f_{r})\\
&=A(f_{0})\cup \cdots \cup A(f_{r})=A(f).
\end{align*}

If the $T_{a}\circ f$-orbit of a point $x\in X(\QQ)$ is Zariski dense, then by Lemma \ref{product case} and Lemma \ref{core case}, we have
\[
\alpha_{f}(x)=\max\{\delta_{h_{i}}=\delta_{f_{i}} \mid \text{$F_{i}$ is not a cyclotomic polynomial}\}=\delta_{f}.
\]
\end{proof}

\begin{proof}[Proof of Theorem \ref{main2 semiab}]
Since $F(1)\neq 1$, by Lemma \ref{minimal polynomial and conjugation}, there exists a point $b\in X(\QQ)$ such that
$T_{-b}\circ f \circ T_{b}$ is a group homomorphism.
Thus it is enough to prove the equivalence of (1), (2) and (3) for every group homomorphism $f$ and $b=0$.
$(3) \Rightarrow (2)$. This follows from the fact that the set of $n$-torsion points of $X$ is finite for each $n>0$
and that the image of an $n$-torsion point by a group homomorphism is also an $n$-torsion point.
$(2) \Rightarrow (1)$ is trivial.
To prove $(1) \Rightarrow (3)$, let $x\in X(\QQ)$ be a point with $ \alpha_{f}(x)=1$.
By \S3, we may assume that the minimal polynomial of $f$ is the form of $F(t)^{e}$ where
$F$ is an irreducible monic polynomial that is not cyclotomic.
We use the notation of Lemma \ref{core case}.
By Theorem \ref{main2} and the inequality $ \alpha_{f}(x)\geq \alpha_{g}(p(x))$, $p(x)$ is a torsion point.
Take $n>0$ so that $np(x)=0$.
Then $nx \in T$.
By Lemma \ref{multiple by n}, $ \alpha_{f_{T}}(nx)=\alpha_{f}(nx)= \alpha_{f}(x)=1$.
Since the minimal polynomial of $f_{T}$ divides $F(t)^{e}$, we can use \cite[Proposition 21(d)]{sil} and
have $nx\in T(\QQ)_{\rm tors}$.
Hence $x\in X(\QQ)_{\rm tors}$.
\end{proof}

\begin{ack}
Part of this paper was written during a research stay at Brown University, to which
the authors are grateful for their hospitality.
Professor Joseph H. Silverman kindly supported the authors' stay at Brown University.
The authors would like to thank Fei Hu for letting us know \cite[Remark 3.4]{dn}.
The authors also would like to thank Yitzchak E. Solomon for his comments.
The first author is supported by the Program for Leading Graduate Schools, MEXT, Japan.
The second author is supported by the Top Global University project for Kyoto University.
\end{ack}

\end{document}